\documentclass[reqno,a4paper,12pt]{amsart} 

\usepackage{amsmath,amscd,amsfonts,amssymb}
\usepackage{mathrsfs,dsfont}

\numberwithin{equation}{section}

\def\R{\mathbb{R}}

\def\Z{\mathbb{Z}}

\def\E{\mathscr{E}}
\def\F{\mathscr{F}}

\def\lam{\lambda}
\def\Lam{\Lambda}

\def\1{\mathds{1}}
\def\eps{\varepsilon}

\renewcommand\ge{\geqslant}
\renewcommand\leq{\leqslant}
\renewcommand\geq{\geqslant}

\renewcommand\hat{\widehat}

\newcommand{\ft}[1]{\widehat #1}
\newcommand{\dotprod}[2]{\langle #1 , #2 \rangle}

\newcommand{\dist}{\operatorname{dist}}

\newcommand{\Vol}{\operatorname{Vol}}

\theoremstyle{plain}
\newtheorem{thm}{Theorem}[section]
\newtheorem{lem}[thm]{Lemma}

\newtheorem*{claim*}{Claim}

\newcommand{\thmref}[1]{Theorem~\ref{#1}}

\newcommand{\lemref}[1]{Lemma~\ref{#1}}

\theoremstyle{definition}

\newtheorem*{definition*}{Definition}
\newtheorem*{remarks*}{Remarks}
\newtheorem*{remark*}{Remark}
\newtheorem{remark}[thm]{Remark}

\newenvironment{enumerate-math}
{\begin{enumerate}
\addtolength{\itemsep}{5pt}
}
{\end{enumerate}}

\newenvironment{enumerate-text}
{\begin{enumerate}
\addtolength{\itemsep}{5pt}
}
{\end{enumerate}}

\begin{document}

 \title[Spectrality of polytopes and equidecomposability]
{Spectrality of polytopes and equidecomposability by translations}

\author{Nir Lev}
\address{Department of Mathematics, Bar-Ilan University, Ramat-Gan 5290002, Israel}
\email{levnir@math.biu.ac.il}

\author{Bochen Liu}
\address{Department of Mathematics, The Chinese University of Hong Kong, Shatin, N.T., Hong Kong}
\email{Bochen.Liu1989@gmail.com}

\thanks{N.L.\ is supported by ISF grant No.\ 227/17 and ERC Starting Grant No.\ 713927.}
\thanks{B.L.\ is partially supported by the grant CUHK24300915 from the Hong Kong Research Grant Council.}

\date{October 31, 2019}


\begin{abstract}
Let $A$ be a polytope in $\R^d$  (not necessarily convex or connected).
We say that $A$ is spectral if the space $L^2(A)$ has an orthogonal 
basis consisting of exponential functions. A result due to
Kolountzakis and Papadimitrakis (2002) asserts that if $A$ is a spectral 
polytope, then  the total area of the $(d-1)$-dimensional faces of
$A$ on which the outward normal is pointing at a given direction, 
must coincide with the total area of those $(d-1)$-dimensional faces 
on which the outward normal is pointing at the opposite direction.
In this paper, we prove an extension of this result to faces of
all dimensions between $1$ and $d-1$. As a consequence 
we obtain that any spectral polytope $A$ can be dissected into a finite 
number of smaller polytopes, which 
can be rearranged using translations to form a cube. 
\end{abstract}

\maketitle


\section{Introduction} \label{secI1}

\subsection{}
Let $A \subset \R^d$ be a bounded, measurable set of positive Lebesgue 
measure. It is said to be  \emph{spectral} if there exists a countable set 
$\Lambda\subset \R^d$  such that the system of exponential functions
\begin{equation}
	\label{eqI1.1}
	E(\Lambda)=\{e_\lambda\}_{\lambda\in \Lambda}, \quad e_\lambda(x)=e^{2\pi
	i\dotprod{\lambda}{x}},
\end{equation}
is orthogonal and complete in  $L^2(A)$,
  that is, the system is an orthogonal basis for the space.
Such a set  $\Lambda$ is called a \emph{spectrum} for $A$.
The classical example of a spectral set is the unit cube $A = \left[-\frac1{2}, \frac1{2}\right]^d$, for which  the set $\Lam = \Z^d$ serves as a spectrum.

Interest in spectral sets has been inspired for many years by an observation due to Fuglede  \cite{Fug74}, that the notion of spectrality is closely related to another, geometrical  notion -- the tiling by translations. We say that $A$ \emph{tiles the space by translations} 
if there exists a countable set $\Lambda\subset \R^d$ such that the collection 
of sets $\{A + \lam\}$, $\lam \in \Lam$, consisting of translated copies of $A$,
constitutes a partition of $\R^d$ up to measure zero.

Fuglede originally conjectured  that a set $A \subset \R^d$ is spectral 
if and only if it can tile the space by translations. While it is still an open problem whether this
conjecture holds e.g.\ for convex domains\footnote{Note added in proof: After the first version of
this paper was submitted, the Fuglede conjecture for convex domains was settled
in the affirmative, see \cite{LM19}.}  (see \cite{Kol00, IKT01, IKT03, GL17, GL18}),
nowadays we know that the conjecture is not true in general, even if $A$ is 
assumed to be a finite union of cubes  \cite{Tao04}.
Nevertheless, with time it became apparent 
that spectral sets behave in many ways like sets which can tile by translations.
In particular, many results about spectral sets have 
analogous results for sets which can tile, and vice versa.
For example, Fuglede proved in \cite{Fug74} that a set
$A$ tiles the space  with respect to a \emph{lattice} translation set $\Lambda$
if and only if the dual lattice $\Lambda^*$ is a spectrum for $A$.

\subsection{}
In this paper we establish a connection between spectrality, and a geometrical notion which is closely related to tiling -- the \emph{equidecomposability by translations}. In this context, we will assume the set $A$ to be a polytope, although not necessarily a convex or a connected one.

Recall that  a \emph{polytope} in $\R^d$ is a set which can be represented as the 
union of a finite number of simplices with disjoint interiors, where a \emph{simplex}
  is the convex hull of $d + 1$ points in $\R^d$  which do not all lie in some hyperplane. 

If $A$ and $B$ are two polytopes in $\R^d$, then they are said to be
 \emph{equidecomposable} (or \emph{dissection equivalent}, or 
\emph{scissors congruent}) if the polytope $A$  can be partitioned,
 up to measure zero, into a finite number of smaller polytopes which can be 
rearranged using rigid motions to form, again up to measure zero, a 
partition of the polytope $B$. If the pieces of the partition can be
rearranged  using translations only, then we say that
 $A$ and $B$ are  \emph{equidecomposable by translations}.

It has long been known that if a polytope  $A \subset \R^d$ can tile the 
space by translations, then  $A$ must be equidecomposable by translations to a cube
of the same volume. This result was first proved
by {M\"urner} in \cite{Mur75}, and
was later rediscovered in \cite{LM95a}. In this paper, we establish
that the analogous result for spectral sets is true:

\begin{thm}
\label{thmA3.2}
Let $A$ be a polytope in $\R^d$  (not necessarily convex or connected).
If $A$ is spectral, then $A$ is equidecomposable by translations
to a cube of the same volume.
\end{thm}

This result can be understood informally as saying that a 
spectral polytope $A \subset \R^d$ can ``nearly''
tile the space by translations. This conclusion is best possible
in a sense, since there are examples of spectral polytopes
which cannot tile (as shown in \cite{Tao04}).

One can easily verify that equidecomposability by translations
constitutes an equivalence relation on the set of all polytopes in $\R^d$.
\thmref{thmA3.2} yields the conclusion that all the spectral polytopes
of a given volume  lie in the same equivalence class.

We will obtain \thmref{thmA3.2} as a consequence of another result,
 which will  also be proved in this paper, and which will be described next.

\subsection{}
In \cite{KP02}, Kolountzakis and Papadimitrakis proved  the
following result: Let $A$ be a polytope in $\R^d$ (again,
$A$ may be non-convex or even disconnected). If $A$ 
 is  spectral,  then  the total area of the $(d-1)$-dimensional faces of
$A$ on which the outward normal is pointing at a given direction, 
must coincide with the total area of those $(d-1)$-dimensional faces 
on which the outward normal is pointing at the opposite direction.

In this paper, we will prove an extension of this result to faces of
all dimensions between $1$ and $d-1$. The statement of our result 
involves certain functions which are called the \emph{Hadwiger functionals}, 
and whose definition will now be given.
For more details we refer the reader to
\cite[Sections 2.10, 3.19]{Bol78} where a friendly introduction to
Hadwiger functionals in dimensions two and three can be found.

Let $r$ be an integer, $1 \leq r \leq d-1$, and suppose that
\begin{equation}
\label{eq:subsps}
V_r \subset V_{r+1} \subset \cdots \subset V_{d-1} \subset V_d = \R^d 
\end{equation}
is a sequence of linear subspaces such that $V_j$ has dimension $j$.
 Each subspace $V_j$ ($r \leq j \leq d-1$)
 in the sequence divides the next one $V_{j+1}$ into two half-spaces; let us call one of them
 the positive half-space, and the other one the negative half-space. Such a sequence of
 nested linear subspaces, endowed with a choice of positive and negative half-spaces, 
will be called an \emph{$r$-flag}, and will be denoted by $\Phi$.

Now let $A$ be a polytope in $\R^d$, and suppose that $A$ has a sequence of faces 
\begin{equation}
\label{eq:faces}
F_r \subset F_{r+1} \subset \cdots \subset F_{d-1} \subset F_d = A,
\end{equation}
where $F_j$ is a $j$-dimensional face of $A$ which is parallel 
to $V_j$ ($r \leq j \leq d-1$).
To each  face $F_j$  we associate a coefficient $\varepsilon_j$, defined in
the following way:
 $\varepsilon_j =+1$  if the face  $F_{j+1}$ adjoins its subface $F_j$ from the
same side where the positive half-space of $V_{j+1}$ adjoins $V_j$; while
$\varepsilon_j = -1$ if $F_{j+1}$ adjoins $F_j$ from the opposite
side. We then define 
\begin{equation}
\label{eq:weight}
H_{\Phi}(A) = \sum \varepsilon_{r} \varepsilon_{r+1} \cdots \varepsilon_{d-1} \Vol_r (F_r) ,
\end{equation}
where the sum goes through all sequences of faces of $A$ as above, and 
where $\Vol_r(F_r)$ denotes the $r$-dimensional volume of $F_r$. 
If no sequence of faces of $A$ as above exists, then we define  the value of $H_{\Phi}(A)$ 
to be zero. We call $H_{\Phi}$ the \emph{Hadwiger functional} associated to the $r$-flag $\Phi$.

For example, if $\Phi$ is a $(d-1)$-flag, then the value of
$H_{\Phi}(A)$ is equal to the difference between 
the total area of the $(d-1)$-dimensional faces of
$A$ on which the outward normal is perpendicular to the hyperplane 
$V_{d-1}$ and is pointing at the direction of the negative half-space 
determined  by $V_{d-1}$, and the total 
area of those $(d-1)$-dimensional faces 
on which the outward normal is pointing at the opposite direction.
Hence the result from \cite{KP02} can be equivalently stated by saying
that if $A$ is spectral, then we must have $H_{\Phi}(A) = 0$
for every $(d-1)$-flag $\Phi$.

We will prove that much more is actually true. Our main result
is the following:

\begin{thm}
\label{thmA3.1}
Let $A$ be a polytope in $\R^d$  (not necessarily convex or connected).
If $A$ is spectral, then $H_{\Phi}(A) = 0$
for every $r$-flag $\Phi$ $(1 \leq r \leq d-1)$.
\end{thm}

This theorem thus extends the result in \cite{KP02} to
$r$-dimensional faces of $A$, for every $r$ between $1$ and $d-1$.

\subsection{}
In the special case when the polytope $A$ is convex, the result in
\cite{KP02} says that if $A$ is spectral, then
each one of the $(d-1)$-dimensional faces of $A$ has a parallel 
face of the same area. By a classical theorem of Minkowski, this 
condition is equivalent to $A$ being centrally symmetric. Hence
any spectral convex polytope must be centrally symmetric. This
result was obtained for the first time in \cite{Kol00}, using a different method.

Moreover, in \cite[Section 4]{GL17} it was proved 
that if a convex, centrally symmetric polytope $A$ is  spectral,
then all the $(d-1)$-dimensional faces of $A$ 
must also be centrally symmetric. This
conclusion can also be stated in terms of the Hadwiger functionals;
indeed, it is equivalent to the statement 
that $H_{\Phi}(A) = 0$ for every $(d-2)$-flag $\Phi$.

In fact, in  \cite[Section 3.3]{Mur77} it is
shown that for a convex polytope $A \subset \R^d$, 
the condition that $H_{\Phi}(A) = 0$ 
for every $r$-flag $\Phi$ $(1 \leq r \leq d-1)$,
is equivalent to $A$ being centrally symmetric and
having centrally symmetric $(d-1)$-dimensional faces.
Thus one can view  \thmref{thmA3.1} as an extension
to non-convex polytopes of the result which states that
if a convex polytope $A$ is  spectral, then $A$ must
be centrally symmetric and have 
centrally symmetric $(d-1)$-dimensional faces.

Our proof of \thmref{thmA3.1} is inspired by both 
\cite{KP02} and \cite[Section 4]{GL17}.
The proof involves an application of a Stokes-type theorem,
which provides an expansion of the Fourier transform
$\ft{\1}_A$ of the indicator function $\1_A$ of a polytope $A \subset \R^d$
 in terms of the Fourier transforms of $r$-dimensional
volume measures on $r$-dimensional faces of $A$.
By identifying the main terms versus error terms in this expansion, we obtain
an approximate expression for the function $\ft{\1}_A$ which is valid
in certain directions. The analysis gets more involved for smaller values of
the face dimension $r$, since then there exist more different types of errors terms, and for
each type a different estimate is required in order to show that the
 term is small.

\subsection{}
We will now clarify the relationship between our two results 
stated above,  namely,  Theorems \ref{thmA3.2}
and \ref{thmA3.1}. In fact, we will see that the first result
 is a consequence of the second one.

We start by recalling that the
 theory of equidecomposability of polytopes originated from
\emph{Hilbert's third problem} -- one of the famous 23 problems posed by
Hilbert at the International Congress of Mathematicians in 1900.
It is obvious that if two polytopes $A$ and $B$ are equidecomposable, then they must
have the same volume. Hilbert's third problem was concerned with the converse assertion:
if $A$ and $B$ are two  polytopes of the same volume, are they necessarily
equidecomposable by rigid motions? It has been  known earlier that in  two dimensions, 
any two polygons of equal area are equidecomposable. However, in the same year 1900
it was shown by Dehn that in three dimensions, such a result is no longer true
 (a comprehensive exposition can be found in \cite{Bol78}).

Dehn's solution to Hilbert's third problem involved an important notion in
the theory of equidecomposability  -- the notion of \emph{additive invariants}. 
Let $G$ be a group of rigid motions of $\R^d$. A function $\varphi$, defined on the set of all polytopes in 
$\mathbb R^d$, is said to be an \emph{additive $G$-invariant} if (i) it is additive, namely, if $A$ and $B$
 are two polytopes with disjoint interiors then $\varphi(A \cup B) = \varphi(A) + \varphi(B)$; 
and (ii) it is invariant under motions from the
group $G$, that is, $\varphi(A) = \varphi(g(A))$ whenever $A$ is a polytope and $g \in G$.

It is obvious that for two polytopes $A$ and $B$ to be equidecomposable using motions
from  $G$, it is  necessary  
that $\varphi(A) = \varphi(B)$ for any additive $G$-invariant $\varphi$. A general problem 
is to construct a ``complete system'' of additive $G$-invariants, that is, invariants which together provide a 
condition which is both necessary and sufficient for two polytopes of the same volume to be 
equidecomposable using motions from the group $G$.

In his solution to Hilbert's third problem, Dehn constructed an additive invariant with respect to the group of all 
rigid motions of $\R^3$, which allowed him to show that a regular tetrahedron and a cube of the
same volume are not equidecomposable \cite{Deh01}. Dehn invariants for polytopes in $\R^d$ have also been studied
\cite{Had54}, and shown to form a complete system in dimensions $d=3,4$ \cite{Syd65, Jes72}. 
It remains an open problem as to whether these invariants are complete also in dimensions $d \geq 5$.

Equidecomposability with respect to the \emph{group of translations}  was
first studied by Hadwiger. He introduced the Hadwiger functionals
$H_\Phi$ defined above, and proved that they form a system of additive invariants 
with respect to translations \cite{Had52, Had57}.
 Moreover, it was shown that the Hadwiger invariants form a 
complete system, so that together they provide a necessary and sufficient condition 
for two polytopes of the same volume to be equidecomposable by translations.
 This was proved by Hadwiger and Glur in dimension two \cite{HG51}, 
by Hadwiger in dimension three \cite{Had68}, and by Jessen and Thorup \cite{JT78}, and independently 
Sah \cite{Sah79}, in every dimension.

This clarifies why \thmref{thmA3.2}  is a consequence of  \thmref{thmA3.1}.
Indeed,   \thmref{thmA3.1} asserts that if a polytope $A \subset \R^d$
is spectral, then we must have $H_{\Phi}(A) = 0$ for every $r$-flag $\Phi$ $(1 \leq r \leq d-1)$. Let $B$ be a cube of the same volume as $A$, then it 
 is easy to check that also $H_{\Phi}(B) = 0$ for every flag $\Phi$.
We thus obtain that $H_{\Phi}(A) = H_{\Phi}(B)$ for all flags $\Phi$.
By the completeness of the Hadwiger invariants we can therefore conclude that
$A$ and $B$ must be equidecomposable by translations, and so
\thmref{thmA3.2} follows.

We remark that the proof given in \cite{Mur75} (or in \cite{LM95a})
 of the fact  that a polytope  $A \subset \R^d$ which can tile by translations 
 must be equidecomposable by translations to a cube, relies on the same consideration.
First it is proved that the tiling assumption implies that
$H_{\Phi}(A) = 0$ for all flags $\Phi$, and then
the completeness of the Hadwiger invariants is used to
conclude that $A$ is  equidecomposable by translations to a cube.

The rest of the paper is devoted to the proof of  \thmref{thmA3.1}.


\section{Preliminaries}
\label{sec:prelim}

\subsection{Notation}
We will use $\dotprod{\cdot}{\cdot}$ and $| \cdot |$ 
to denote respectively the standard scalar product and norm in $\R^d$.
We denote by $\vec e_1, \vec e_2, \dots, \vec e_d$ the standard basis vectors in $\R^d$,
and by $x_1, x_2, \dots, x_d$ the coordinates 
of a vector $x \in \R^d$.

If $A \subset \R^d$ and $\tau$ is a vector in $\R^d$,
then we let $A + \tau = \{a+ \tau: a \in A\}$ denote
 the translate of $A$ by the vector
$\tau$. If $A, B$ are two subsets of $\R^d$, then
$A+B$ and $A-B$ denote respectively their set of sums and set of 
differences.

For each $\xi \in \R^d$ we denote by $e_\xi$ the exponential function
$e_\xi (x) := e^{2\pi i\langle \xi,x\rangle}$,  $x \in \R^d$.

By the Fourier transform of a function $f \in L^1(\R^d)$ we mean the function
\[
\ft f (\xi)=\int_{\R^d} f (x) \, \overline{e_\xi (x)} \, dx,
\]
and similarly, the Fourier transform of a finite, complex measure
 $\mu$ on $\R^d$ is  the function
\[
\ft{\mu} (\xi)=\int_{\R^d} \overline{e_\xi (x)}  \,  d\mu(x).
\]

\subsection{Spectra}
If $A$ is a bounded, measurable set in $\R^d$ of positive measure, then
by a \emph{spectrum} for $A$ we mean a countable set $\Lambda\subset\R^d$
such that the system of exponential functions $E(\Lam)$ defined by \eqref{eqI1.1} is 
orthogonal and complete in the space $L^2(A)$.

For any two points $\lam,\lam'$ in $\R^d$ we have
$\dotprod{e_\lambda}{e_{\lambda'}}_{L^2(A)} = \hat{\1}_A(\lambda'-\lambda)$, where
$\ft{\1}_A$ is the Fourier transform of the indicator function $\1_A$ of the set $A$.
The orthogonality of the system $E(\Lambda)$ in $L^2(A)$ is therefore equivalent to the condition
\begin{equation}
	\label{eqP1.2}
	(\Lambda-\Lambda) \setminus \{0\} \subset \{ \xi \in \R^d : \ft{\1}_A(\xi) = 0\}.
\end{equation}

A set $\Lambda\subset \R^d$ is said to be \emph{uniformly discrete} if there is
$\delta>0$ such that $|\lambda'-\lambda|\ge \delta$ for any two distinct points
$\lambda,\lambda'$ in $\Lambda$. 
The condition \eqref{eqP1.2} implies that every spectrum $\Lambda$ of $A$
is a uniformly discrete set.

The set $\Lambda$ is said to be \emph{relatively dense} if
there is $R >0$ such that every ball of radius $R$ contains at least one point from $\Lam$. It is well-known that if $\Lam$ is a spectrum for $A$,
then $\Lam$ must also be a relatively dense 
set (see e.g.\ \cite[Section 2C]{GL17}).

The property of $\Lambda$ being a spectrum for $A$ is
invariant under translations of both $A$ and $\Lambda$.
If $M$ is a $d \times d$  invertible matrix, then
 $\Lambda$ is a spectrum for $A$ if and only if
the set $(M^{-1})^\top (\Lambda)$ is a spectrum for $M(A)$.

\subsection{Polytopes and equidecomposability}
A \emph{simplex} in $\R^d$ is the convex hull of $d + 1$ points
 which do not all lie in some hyperplane. A \emph{polytope} in $\R^d$
is a set which
can be represented as the union of a finite number of simplices
with disjoint interiors. Remark that a polytope is not necessarily
a convex, nor even a connected, set.

Let $A$ and $B$ be two polytopes in $\R^d$. We say that
$A$ and $B$ are \emph{equidecomposable}
if there exist finite decompositions of $A$ and $B$ of the form
\[
A = \bigcup_{j=1}^{N} A_j, \quad  B = \bigcup_{j=1}^{N} B_j
\]
where $A_1, \dots, A_N$ are polytopes with pairwise disjoint interiors,
$B_1, \dots, B_N$ are also polytopes with pairwise disjoint interiors,
and for each $j$ the polytope $B_j$ is the image of $A_j$ under some
 rigid motion.
If for each $j$ there is a vector $\tau_j \in \R^d$
such that $B_j = A_j + \tau_j$ (that is, 
$B_j$ is the image of $A_j$ under  translation), then we say that the polytopes
$A$ and $B$ are \emph{equidecomposable by translations}.

\subsection{Flags}
If $r$ is an integer, $0 \leq r \leq d-1$, then an \emph{$r$-flag} $\Phi$ in $\R^d$ 
is defined to be a sequence of linear subspaces 
\begin{equation}
\label{eq:C3.1.1}
V_r \subset V_{r+1} \subset \cdots  \subset V_{d-1} \subset V_d = \R^d
\end{equation}
such that $V_j$ has dimension $j$. Each subspace $V_j$ ($r \leq j \leq d-1$)
 in the sequence divides the next one $V_{j+1}$ into two half-spaces;
we assume that $\Phi$ is endowed with a choice of one of these
half-spaces being called positive, and the other being called negative.

It will be convenient to define also a \emph{$d$-flag} in $\R^d$ to be the sequence 
which consists of just one subspace $V_d = \R^d$.

Let $A$ be a polytope in $\R^d$, and suppose that we have a sequence
\[
F_r \subset F_{r+1} \subset \cdots \subset F_{d-1} \subset F_d = A,
\]
where $F_j$ is a $j$-dimensional face of $A$ $(r \leq j \leq d-1)$.
 Such a sequence will be called an \emph{$r$-sequence} of faces of 
the polytope $A$, and will be denoted by $\F_r$.

Let $\Phi$ be an $r$-flag  determined by
a sequence of linear subspaces 
$V_r \subset V_{r+1} \subset \cdots  \subset  V_d$,
and let $\F_r$ be an $r$-sequence of faces 
$F_r \subset F_{r+1} \subset \cdots \subset  F_d$
of $A$. We say that the face $F_j$ is parallel to the
subspace $V_j$ if the affine hull of $F_j$ is a translate of $V_j$.
We say that the $r$-sequence $\F_r$ is parallel to the
 $r$-flag $\Phi$ if $F_j$ is parallel to $V_j$ for each 
$r \leq j \leq d-1$.

Each $r$-flag $\Phi$ $(1 \leq r \leq d-1)$ determines a function 
$H_{\Phi}$ defined on the set of all polytopes in $\R^d$, which is given 
by \eqref{eq:weight}. The function $H_\Phi$ is  additive,
and it is invariant with respect to translations. It will be called the
\emph{Hadwiger functional} associated to the $r$-flag $\Phi$.

Notice that if two $r$-flags $\Phi$ and $\Psi$ correspond to the same sequence of linear subspaces
 $V_r \subset V_{r+1} \subset \cdots \subset V_d$, then 
either $H_{\Phi} = H_{\Psi}$  or $H_{\Phi} = - H_{\Psi}$
(depending on the choice of positive and negative half-spaces).
Hence each sequence of linear subspaces essentially corresponds to one Hadwiger functional.

If $\Phi$ is a $d$-flag, then its associated Hadwiger functional $H_\Phi$
is defined by  $H_\Phi(A)=\Vol_d(A)$ for any polytope $A \subset \R^d$.

(We do not consider Hadwiger functionals associated to $0$-flags, 
as these functionals vanish identically and thus they do not provide any information.)

\subsection{Flag measures}
Let $\Phi$ be an $r$-flag   in $\R^d$ $(0 \leq r \leq d)$,
determined by a sequence of linear subspaces  \eqref{eq:C3.1.1}.
To each polytope $A\subset \R^d$ we associate a signed measure 
$\mu_{A,\Phi}$ on $\R^d$ given by
\begin{equation}
\label{eq:C3.1.2}
\mu_{A,\Phi} = \sum_{\F_r} \varepsilon_{r} \varepsilon_{r+1} \cdots \varepsilon_{d-1} \, {\Vol_r}|_{F_r} ,
\end{equation}
where $\F_r$  goes through all $r$-sequences of faces
$F_r \subset F_{r+1} \subset \cdots \subset F_d$
of the polytope $A$ that are parallel to $\Phi$,
 the $\varepsilon_j$ are the $\pm 1$ coefficients 
associated to the $r$-sequence $\F_r$ with
respect to $\Phi$ in the same way as in \eqref{eq:weight},
 and  ${\Vol_r}|_{F_r}$ denotes the $r$-dimensional volume measure
restricted to the face $F_r$.

If $r=0$, then by an $r$-dimensional face of $A$ we mean a 
vertex of $A$, and by the measure ${\Vol_r}|_{F_r}$ we mean
the Dirac measure at the vertex $F_r$.
Hence the flag measure $\mu_{A,\Phi}$ associated to a $0$-flag $\Phi$ is a
discrete measure supported on vertices of $A$.

If $\Phi$ is a $d$-flag, then $\mu_{A,\Phi} = {\Vol_d}|_A$
(the Lebesgue measure restricted to $A$).

It follows from \eqref{eq:weight} and \eqref{eq:C3.1.2} that the 
measure $\mu_{A,\Phi}$ satisfies
\begin{equation}
\label{eq:C3.1.5}
\int d \mu_{A,\Phi}  = H_{\Phi}(A)
\end{equation}
for any $r$-flag $\Phi$   $(1 \leq r \leq d)$.

(If $\mu_{A,\Phi}$ is the flag measure associated to a $0$-flag $\Phi$,
then $\int d \mu_{A,\Phi}  = 0$.)


\section{Stokes-type theorem for Fourier transforms of flag measures}

The main result obtained in this section (\thmref{thmC7.1})
provides an expansion of the Fourier transform of a $k$-dimensional
flag measure, in terms of  Fourier transforms of $(k-1)$-dimensional
flag measures. It is basically an application of Stokes theorem,
which allows us to replace integration over $k$-dimensional
faces of a polytope, by integration over the relative
boundaries of these faces (see also \cite[p.\ 341]{Bar02}, for instance).

In \cite[Section 4]{LL18} we proved a similar result but in a
more refined context, where the equidecomposability of polytopes
was studied with respect to a proper subgroup of all the translations.
For the completeness of our exposition,
we reproduce here the arguments in a self-contained 
version that is suitable for our present context.

\subsection{}
Let $A$ be a polytope in $\R^d$, and let $\Phi_k$ be a $k$-flag $(1 \leq k \leq d)$ 
determined by a sequence of linear subspaces 
$V_k \subset V_{k+1} \subset \dots \subset V_d$. 
The Fourier transform of the measure $\mu_{A,\Phi_k}$ is given by
\begin{equation}
\label{eq:C7-ft-muj}
\ft{\mu}_{A,\Phi_k}(\xi) = \int \overline{e_\xi} \, d\mu_{A,\Phi_k} =
\sum_{\F_k}  \eps_k \eps_{k+1} \cdots\eps_{d-1}\int_{F_k} \overline{e_\xi},
\end{equation}
where  $\F_k$ goes through all $k$-sequences of faces
$F_k \subset F_{k+1} \subset \cdots \subset F_d$
of the polytope $A$ that are parallel to $\Phi_k$, 
the $\varepsilon_j$'s are the $\pm 1$ coefficients associated to the $k$-sequence
$\F_k$ with respect to $\Phi_k$, and the integral on the right hand side
is taken with respect  to the $k$-dimensional volume measure on the face $F_k$.

Let $\partial F_k$ denote the relative boundary of the face $F_k$, and for each
$x \in \partial F_k$ let $n(x)$ be a vector in the linear subspace 
$V_k$ which is outward unit normal  to $F_k$ at the point $x$.
Then for every $v \in V_k$ we have
\begin{equation}
\label{eq:C7-div-thm}
-2\pi i \dotprod{\xi}{v} 
\int_{F_k} \overline{e_\xi} = \int_{\partial F_k} \dotprod{n}{v} \, \overline{e_\xi},
\end{equation}
which follows by applying the divergence theorem to the 
function $f(x) =  \overline{e_\xi (x)} \, v$ over the face $F_k$.
The relative boundary $\partial F_k$ consists of a finite number
of $(k-1)$-dimensional faces $F_{k-1}$ of $F_k$.
Hence, using \eqref{eq:C7-ft-muj} and \eqref{eq:C7-div-thm}, we get
\begin{align}
\label{eq:C7-sum-in-r-1-flags-a}
&-2\pi i \dotprod{\xi}{v} \, \ft{\mu}_{A,\Phi_k}(\xi) 
= \sum_{\F_k}  \eps_k \eps_{k+1}  \cdots\eps_{d-1} \int_{\partial F_k}
 \dotprod{n}{v} \, \overline{e_\xi} \\
\label{eq:C7-sum-in-r-1-flags-b}
&\qquad =  \sum_{\F_k} \eps_k \eps_{k+1}  \cdots\eps_{d-1}
 \sum_{F_{k-1}} \dotprod{n}{v} \int_{ F_{k-1}} \, \overline{e_\xi},
\end{align}
where $F_{k-1}$ goes through the $(k-1)$-dimensional subfaces of 
the $k$-dimensional face $F_k$ from  the sequence $\F_k$,
 and $n$ is the outward unit normal to $F_k$ on $F_{k-1}$.

Let $\E$ be the collection of all the $(k-1)$-sequences  of faces 
$F_{k-1} \subset F_k \subset \cdots \subset F_d$ of
 the polytopes $A$,  such that $F_j$ is parallel to $V_j$ 
$(k \leq j \leq d-1)$. 
We define an equivalence relation on $\E$ by saying that two elements 
$\F_{k-1}$ and $\F'_{k-1}$ from $\E$ are equivalent
if the $(k-1)$-dimensional face $F_{k-1}$ from the sequence $\F_{k-1}$ 
is parallel to the $(k-1)$-dimensional face 
$F'_{k-1}$ from $\F'_{k-1}$. Then $\E$ can be partitioned  into a finite number of
equivalence classes $\E^1, \E^2, \dots, \E^N$ induced
by this  equivalence relation.

To each equivalence class $\E^l$ $(1 \leq l \leq N)$ we associate a 
$(k-1)$-flag $\Phi_{k-1}^l$, 
defined in the following way. The flag $\Phi_{k-1}^l$  is 
determined by a sequence of linear subspaces  
\[
V_{k-1}^l \subset V_k \subset V_{k+1} \subset \cdots  \subset V_d = \R^d,
\]
where $V_k, V_{k+1}, \dots, V_d$ are the  linear subspaces  that 
determine the $k$-flag $\Phi_k$, while $V^l_{k-1}$ is a new linear subspace of
dimension $k-1$. The subspace $V^l_{k-1}$  is chosen such that it is
parallel to all the $(k-1)$-dimensional faces $F_{k-1}$ belonging to
sequences  $\F_{k-1}$ from the equivalence class $\E^l$.
It is obvious from the definition of the equivalence relation on $\E$
that the subspace $V^l_{k-1}$ exists and that it is unique.
 We endow the $(k-1)$-flag $\Phi_{k-1}^l$ with a choice of positive 
and negative half-spaces, 
by saying that the positive and negative half-spaces of $V_{j+1}$ determined
by the subspace $V_j$ coincide with those from the $k$-flag $\Phi_k$
for all $k \leq j \leq d-1$; while the positive and negative half-spaces 
of $V_k$ that are determined by the new subspace $V^l_{k-1}$
are selected in an arbitrary way.

For each $1 \leq l \leq N$, let $\sigma^l$ denote the (unique) unit vector in 
the linear subspace $V_k$ which is normal to $V^l_{k-1}$ and is pointing
 towards the negative half-space of $V_k$ determined  by $V^l_{k-1}$. 
We then observe that if $\F_{k-1}$ is a sequence of faces
$F_{k-1} \subset F_k \subset \cdots \subset F_d$
belonging to the equivalence class $\E^l$, and if
$n$ is the outward unit normal to $F_k$ on $F_{k-1}$,
then we have $n = \eps_{k-1} \sigma^l$, where 
$\eps_{k-1} = +1$  if $F_k$ adjoins $F_{k-1}$ from the positive side
of $V_k$ which is determined by $V^l_{k-1}$,  and
$\eps_{k-1} = -1$  if $F_k$ adjoins $F_{k-1}$ from the negative side.
It follows that the sum in
\eqref{eq:C7-sum-in-r-1-flags-b} is equal to
\begin{equation}
\label{eq:C7-sum-in-r-1-flags-c}
\sum_{l=1}^{N} \dotprod{\sigma^l}{v} \sum_{\F_{k-1}} 
 \eps_{k-1} \eps_k \eps_{k+1}  \cdots\eps_{d-1}\int_{ F_{k-1}} \, \overline{e_\xi},
\end{equation}
where  $\F_{k-1}$ goes through all  $(k-1)$-sequences of faces
$F_{k-1} \subset F_k \subset \cdots \subset F_d$
of the polytope $A$ that are parallel to $\Phi_{k-1}^l$, and
the $\varepsilon_j$'s are the $\pm 1$ coefficients associated to the $(k-1)$-sequence
$\F_{k-1}$ with respect to $\Phi_{k-1}^l$. 
But now the inner sum in \eqref{eq:C7-sum-in-r-1-flags-c} is just
the integral of the function $\overline{e_\xi}$ with respect to the measure
$\mu_{A,\Phi_{k-1}^l}$. Hence combining
\eqref{eq:C7-sum-in-r-1-flags-a}, \eqref{eq:C7-sum-in-r-1-flags-b},
\eqref{eq:C7-sum-in-r-1-flags-c} we finally arrive at the following result:

\begin{thm}
\label{thmC7.1}
Let $A$ be a polytope in $\R^d$, and let
$\Phi_k$ be a $k$-flag $(1 \leq k \leq d)$
determined by a sequence of linear subspaces 
$V_k \subset V_{k+1} \subset \dots \subset V_d$. Then 
for every $\xi \in \R^d$ and every
$v\in V_k$ we have
\begin{equation}
\label{eq:C7.1.1}
-2 \pi i \dotprod{\xi}{v} \, \ft{\mu}_{A,\Phi_k} (\xi) = 
\sum_{l=1}^{N} \dotprod{\sigma^l}{v} \,
\ft{\mu}_{A,\Phi_{k-1}^l} (\xi),
\end{equation}
where the flags $\Phi_{k-1}^l$ 
and vectors $\sigma^l$   are as above.
\end{thm}

\begin{remark}
It may happen that the polytope $A$ does not have any $k$-sequences
  of faces $\F_k$ that are parallel to the $k$-flag $\Phi_k$. In this case, 
$\mu_{A,\Phi_{k}}$ is the zero measure, 
and the right hand side of \eqref{eq:C7.1.1} is understood
to be an empty sum.
\end{remark}


\section{Asymptotics of Fourier transform}

In this section we use the flag measures $\mu_{A, \Phi}$
to analyze the asymptotic behavior of
the Fourier transform  $\ft{\1}_A$ of the indicator function
  of a polytope $A \subset \R^d$. The main result of
this section (\thmref{thmC3.5}) 
provides approximate expressions for $\ft{\1}_A$
which are valid in certain unbounded domains,
in terms of the Fourier transforms $\ft{\mu}_{A, \Phi}$
of the flag measures.

\subsection{}
Let $\Phi_r$ be an $r$-flag $(0 \leq r \leq d-1)$. We will say that
$\Phi_r$ is in \emph{standard position}  if it is determined by the sequence 
of linear subspaces  $V_r, V_{r+1}, \dots, V_d$  given by
\begin{equation}
\label{eq:C3.9.1}
V_j = \{x \in\R^d : x_{j+1} = x_{j+2} = \dots = x_{d} = 0\}, \quad r \leq j \leq d-1,
\end{equation}
and the positive and negative half-spaces of $V_{j+1}$ that are
determined by $V_j$ are chosen such that
$V_{j+1} \cap \{x: x_{j+1} < 0\}$ is the positive half-space,
 while $V_{j+1} \cap \{x: x_{j+1} > 0\}$ is the negative half-space,
for all $r \leq j \leq d-1$.

Given an integer $0 \leq r \leq d-1$, and three positive real numbers
$\alpha$, $\delta$ and $L$ such that
$0 < 2 \delta < \alpha < 1$, we denote by 
$K(r,\alpha,L,\delta)$ the set of all vectors $\xi \in \R^d$
satisfying the following three conditions:
\begin{equation}
\label{eq:C7.1}
|\xi_j|  \leq \alpha |\xi_{r+1}|   \quad (1 \leq j \leq r),
\end{equation}
\begin{equation}
\label{eq:C4.2}
L  \leq |\xi_{r+1}|,
\end{equation}
\begin{equation}
\label{eq:C4.3}
 |\xi_{j}| \leq 2\delta |\xi_{j+1}|  \quad (r+1 \leq j \leq d-1).
\end{equation}

In this section, our goal is to prove:

\begin{thm}
\label{thmC3.5}
Let $A$ be a polytope in $\R^d$, and let
$\Phi_r$ be an $r$-flag in standard position $(0 \leq r \leq d-1)$.
Then there exists $\alpha > 0$, such that for any $\eta > 0$ 
one can find $\delta$ and $L$ such that
\begin{equation}
\label{eq:C3.5.2}
\Big| 
\Big( \ft{\1}_A (\xi) \prod_{j=r+1}^{d} (-2 \pi i \xi_j)  \Big) -
\ft{\mu}_{A,\Phi_r} (\xi) \Big|  < \eta, \quad \xi \in K(r,\alpha,L,\delta).
\end{equation}
\end{thm}

This result allows us to approximate $\ft{\1}_A$
in the domain $K(r,\alpha,L,\delta)$ in terms of the Fourier transform
of the flag measure  $\mu_{A, \Phi_r}$. This shows that
the behavior of the Fourier transform $\ft{\1}_A$
in the domain $K(r,\alpha,L,\delta)$ is essentially governed only by
the contribution of those $r$-dimensional faces $F_r$
of $A$ that belong to some $r$-sequence $F_{r}, F_{r+1}, \dots, F_d$
of faces which is paraellel to the $r$-flag $\Phi_r$.

Notice that the estimate \eqref{eq:C3.5.2} yields different information
for different values of $r$. Namely, for smaller $r$  we obtain a more accurate 
approximation for the Fourier transform $\ft{\1}_A$, but 
the domain in which this approximation is valid is also smaller.

The requirement in \thmref{thmC3.5}
that the $r$-flag $\Phi_r$ be in standard position, is done merely 
in order to simplify the  notation in the statement. Indeed,
a similar result for an arbitrary $r$-flag (that is, an $r$-flag which is
not necessarily in standard position) can be deduced easily, by
using the fact that any $r$-flag in $\R^d$ can be mapped by an 
invertible linear transformation onto an $r$-flag in standard position.

The rest of the section is devoted to the proof of \thmref{thmC3.5}. 
We divide the proof into a series of lemmas.

\subsection{}

\begin{lem}
\label{lemC3.15}
Let $A$ be a polytope in $\R^d$,  let $0 \leq r \leq d-1$,
and let $\Psi_k$ be a $k$-flag $(1 \leq k \leq d)$
determined by a sequence of linear subspaces 
$W_k \subset W_{k+1} \subset \dots \subset W_d$. 
Let $m$ be the smallest element of the set $\{0,1,2,\dots,d\}$ such that
\begin{equation}
\label{eq:C3.10.2}
W_k \subset \{x \in \R^d: x_{m+1} = x_{m+2} = \dots = x_d = 0 \},
\end{equation}
and suppose that 
\begin{equation}
\label{eq:C3.10.2.1}
m \geq r+1.
\end{equation}
 Then there exist $\alpha > 0$, a constant $C$, 
and $(k-1)$-flags  $\Psi_{k-1}^1, \Psi_{k-1}^2, \dots, \Psi_{k-1}^N$
such that for any $\delta$ and $L$ we have
\begin{equation}
\label{eq:C3.15.1}
\big| (-2 \pi i \xi_m) \ft{\mu}_{A,\Psi_k} (\xi)  \big|
\leq C \sum_{l=1}^{N} |\ft{\mu}_{A,\Psi_{k-1}^l} (\xi)|,
\quad
\xi \in K(r,\alpha,L,\delta).
\end{equation}
\end{lem}

\begin{proof}
Since $W_k$ is a linear subspace of dimension $k$, we must have
$m \geq k$.
Then it follows from the definition of $m$ 
that we can find a vector $v \in W_k$ such that $v_m \neq 0$.
By multiplying $v$ on an appropriate scalar we may assume
that $v_m > 1$.

Let $\xi \in K(r,\alpha,L,\delta)$.  It follows from \eqref{eq:C3.10.2} that
$ v_{m+1} = v_{m+2} = \dots = v_d = 0$, hence
\begin{equation}
\label{eq:C3.10.3}
|\dotprod{\xi}{v}| = \Big|\sum_{j=1}^{m} \xi_j v_j \Big|
\geq  |\xi_m v_m| - \Big|\sum_{j=1}^{m-1} \xi_j v_j\Big|.
\end{equation}
The conditions \eqref{eq:C7.1}, \eqref{eq:C4.3}, \eqref{eq:C3.10.2.1}
  ensure that if we choose $\alpha > 0$ small enough
(in a way that depends on the vector $v$ but does not depend on $\xi$), then
the right hand side of \eqref{eq:C3.10.3} will be not less than
$|\xi_m|$. We thus obtain that
\begin{equation}
\label{eq:C3.10.4}
|\dotprod{\xi}{v}| \geq  |\xi_m |, \quad
\xi \in K(r,\alpha,L,\delta).
\end{equation}

We now apply \thmref{thmC7.1} to the $k$-flag $\Psi_k$ 
and to the vector $v$. The theorem gives
\begin{equation}
\label{eq:C3.10.5}
-2 \pi i \dotprod{\xi}{v} \, \ft{\mu}_{A,\Psi_k} (\xi) = 
\sum_{l=1}^{N} \dotprod{\sigma^l}{v} \,
\ft{\mu}_{A,\Psi_{k-1}^l} (\xi).
\end{equation}
Combining this with \eqref{eq:C3.10.4} and the estimate
$|\dotprod{\sigma^l}{v}| \leq |v|$, implies that \eqref{eq:C3.15.1} holds.
\end{proof}

\subsection{}

\begin{lem}
\label{lemC3.14}
Let $A$ be a polytope in $\R^d$,  and let $\Psi_r$ be an $r$-flag 
$(1 \leq r \leq d-1)$ determined by a sequence of linear subspaces 
$W_r \subset W_{r+1} \subset \dots \subset W_d$. 
Assume that $W_r$ does not coincide with the subspace
\begin{equation}
\label{eq:C3.14.1}
V_r = \{x \in\R^d : x_{r+1} = x_{r+2} = \dots = x_{d} = 0\}.
\end{equation}
Then there exists $\alpha > 0$, such that for any $\eta > 0$ 
one can find $L$ such that
\begin{equation}
\label{eq:C3.14.2}
| \ft{\mu}_{A,\Psi_r} (\xi)  | < \eta, \quad \xi \in K(r,\alpha,L,\delta).
\end{equation}
\end{lem}

\begin{proof}
We wish to apply \lemref{lemC3.15} with $k=r$. Indeed,
the assumption that $W_r$ does not coincide with the subspace
$V_r$ in \eqref{eq:C3.14.1} implies that condition
\eqref{eq:C3.10.2.1} is satisfied, hence we may use 
\lemref{lemC3.15}.  The lemma yields that
the estimate \eqref{eq:C3.15.1} is true, provided that $\alpha > 0$
is sufficiently small and the constant $C$ is sufficiently large.

If $\xi \in K(r,\alpha,L,\delta)$, then \eqref{eq:C4.2}, \eqref{eq:C4.3}
 imply  that $|\xi_m| \geq |\xi_{r+1}| \geq L$.
So from \eqref{eq:C3.15.1} we get
\begin{equation}
\label{eq:C3.14.6}
2 \pi L \, |\ft{\mu}_{A,\Psi_r} (\xi)  |
\leq C \sum_{l=1}^{N} |\ft{\mu}_{A,\Psi_{r-1}^l} (\xi)|,
\quad
\xi \in K(r,\alpha,L,\delta).
\end{equation}
Notice that the right hand side of the inequality
in \eqref{eq:C3.14.6} is bounded as a function of $\xi$.
Hence given $\eta>0$, if we choose $L$ sufficiently large then \eqref{eq:C3.14.2} holds.
\end{proof}

\subsection{}

\begin{lem}
\label{lemC3.7}
Let $A$ be a polytope in $\R^d$,  let $0 \leq r \leq d-1$,
and let $\Psi_k$ be a $k$-flag $(r+1 \leq k \leq d)$.
Then there exist $\alpha > 0$ and a constant $C$, such that for any
$\delta$ and $L$ we have
\begin{equation}
\label{eq:C3.7.1}
\Big| \ft{\mu}_{A,\Psi_k} (\xi) \prod_{j=r+1}^{k} (-2 \pi i \xi_j)  \Big|
\leq C, \quad 
\xi \in K(r,\alpha,L,\delta).
\end{equation}
\end{lem}

\begin{proof}
Again we wish to apply \lemref{lemC3.15}. Since we have
$m \geq k \geq r+1$, the condition
\eqref{eq:C3.10.2.1} is satisfied, and  the
lemma yields that 
the estimate \eqref{eq:C3.15.1} is true, provided that $\alpha > 0$
is sufficiently small and  the constant $C$ is sufficiently large.

If $\xi \in K(r,\alpha,L,\delta)$, then \eqref{eq:C4.3} implies  that $|\xi_m| \geq |\xi_k|$.
Hence \eqref{eq:C3.15.1}  implies that
\begin{equation}
\label{eq:C3.7.6}
\big| (-2 \pi i \xi_k) \ft{\mu}_{A,\Psi_k} (\xi)  \big|
\leq C \sum_{l=1}^{N} |\ft{\mu}_{A,\Psi_{k-1}^l} (\xi)|,
\quad
\xi \in K(r,\alpha,L,\delta).
\end{equation}
We notice that the right hand side of the inequality
in \eqref{eq:C3.7.6} is bounded as a function of $\xi$. 
This confirms that \eqref{eq:C3.7.1} is true in the special
case when $k=r+1$.

It remains to prove \eqref{eq:C3.7.1} also in the case when
$r+2 \leq k \leq d$. This will be done by induction on $k$. 
We multiply each side
of  \eqref{eq:C3.7.6} by the absolute values of the
terms $-2 \pi i \xi_j$ $(r+1 \leq j \leq k-1)$, and obtain
\begin{equation}
\label{eq:C3.7.7}
\Big|
\ft{\mu}_{A,\Psi_k} (\xi) \prod_{j=r+1}^{k} (-2 \pi i \xi_j)  \Big|
\leq C \sum_{l=1}^{N} 
\Big| \ft{\mu}_{A,\Psi_{k-1}^l} (\xi)
\prod_{j=r+1}^{k-1} (-2 \pi i \xi_j)  \Big|.
\end{equation}
By the inductive hypothesis, each one of the terms in the sum
on the right hand side of \eqref{eq:C3.7.7} is bounded in the domain
$K(r,\alpha,L,\delta)$, provided that $\alpha>0$ is sufficiently small.
Hence also the left hand side is bounded, and again we arrive
at \eqref{eq:C3.7.1}.
\end{proof}

\subsection{}

\begin{lem}
\label{lemC3.8}
Let $A$ be a polytope in $\R^d$,  let $0 \leq r \leq d-1$,
and let $\Psi_k$ be a $k$-flag $(r+1 \leq k \leq d)$
determined by a sequence of linear subspaces 
$W_k \subset W_{k+1} \subset \dots \subset W_d$. 
Assume that $W_k$ does not coincide with the subspace
\begin{equation}
\label{eq:C3.8.1}
V_k = \{x \in\R^d : x_{k+1} = x_{k+2} = \dots = x_{d} = 0\}.
\end{equation}
Then there exists $\alpha > 0$, such that for any $\eta > 0$ 
one can find $\delta$ such that 
\begin{equation}
\label{eq:C3.8.2}
\Big| \ft{\mu}_{A,\Psi_k} (\xi)  \prod_{j=r+1}^{k} (-2 \pi i \xi_j)  \Big|
< \eta, \quad \xi \in K(r,\alpha,L,\delta).
\end{equation}
\end{lem}

\begin{proof}
Once more we wish to apply \lemref{lemC3.15}.
The assumption that $W_k$ does not coincide with the subspace
\eqref{eq:C3.8.1} implies that the number $m$ from the lemma
satisfies the condition $m \geq k+1$.
In particular, \eqref{eq:C3.10.2.1} holds and we may apply
the lemma, which yields that
the estimate \eqref{eq:C3.15.1} is true, provided that $\alpha > 0$
is sufficiently small and  the constant $C$ is sufficiently large.

Let $\xi \in K(r,\alpha,L,\delta)$.
Then the conditions $k \geq r+1$ and  $m \geq k+1$ imply,
using \eqref{eq:C4.3}, that $|\xi_m| \geq  (2\delta)^{-1} |\xi_k|$. So it 
follows from \eqref{eq:C3.15.1} that
\begin{equation}
\label{eq:C3.8.6.3}
\big| (-2 \pi i \xi_k) \ft{\mu}_{A,\Psi_k} (\xi)  \big|
\leq 2C \delta \sum_{l=1}^{N} |\ft{\mu}_{A,\Psi_{k-1}^l} (\xi)|.
\end{equation}
The sum on the right hand side is bounded as a function of $\xi$.
Hence given $\eta>0$, if we choose  $\delta>0$ small enough
then we can make the right hand side of \eqref{eq:C3.8.6.3}
smaller than $\eta$ in the domain $K(r,\alpha,L,\delta)$. This yields
\eqref{eq:C3.8.2} in the case when $k=r+1$.

In the case when $r+2 \leq k \leq d$, we multiply each side
of  \eqref{eq:C3.8.6.3} by the absolute values of the
terms $-2 \pi i \xi_j$ $(r+1 \leq j \leq k-1)$, and obtain
\begin{equation}
\label{eq:C3.8.6.2}
\Big|
\ft{\mu}_{A,\Psi_k} (\xi) \prod_{j=r+1}^{k} (-2 \pi i \xi_j)  \Big|
\leq 2 C \delta  \sum_{l=1}^{N} 
\Big| \ft{\mu}_{A,\Psi_{k-1}^l} (\xi)
\prod_{j=r+1}^{k-1} (-2 \pi i \xi_j)  \Big|.
\end{equation}
The sum on the right hand side of \eqref{eq:C3.8.6.2}
is bounded as a function of $\xi$, according to 
\lemref{lemC3.7}. Hence again, given $\eta>0$ we can
choose $\delta>0$  such that \eqref{eq:C3.8.2} holds.
\end{proof}

\subsection{}

\begin{lem}
\label{lemC3.6}
Let  $A$ be a polytope in $\R^d$, and let
$\Phi_r$ be an $r$-flag, and $\Phi_k$ be a $k$-flag $(0 \leq r < k \leq d)$,
both in standard position.
Then there exists $\alpha > 0$, such that for any $\eta > 0$ 
one can find $\delta$ and $L$ such that
\begin{equation}
\label{eq:C3.6.2}
\Big| 
\Big( \ft{\mu}_{A,\Phi_k} (\xi) \prod_{j=r+1}^{k} (-2 \pi i \xi_j)  \Big) -
\ft{\mu}_{A,\Phi_r} (\xi) \Big|  < \eta, \quad \xi \in K(r,\alpha,L,\delta).
\end{equation}
\end{lem}

\begin{proof}
Let $V_r, V_{r+1}, \dots, V_d$ be the linear subspaces   given by
\eqref{eq:C3.9.1}. We apply \thmref{thmC7.1} to the $k$-flag $\Phi_k$ 
and to the vector $v = \vec{e}_k$ which belongs to $V_k$.
Then from \eqref{eq:C7.1.1} we get
\begin{equation}
\label{eq:C3.6.3}
-2 \pi i \xi_k \, \ft{\mu}_{A,\Phi_k} (\xi) = 
\ft{\mu}_{A,\Phi_{k-1}} (\xi) +
\sum_{l=1}^{N} \dotprod{\sigma^l}{\vec{e}_k} \,
\ft{\mu}_{A,\Psi_{k-1}^l} (\xi),
\end{equation}
where $\Phi_{k-1}$ is  a $(k-1)$-flag in standard position,
and each $\Psi_{k-1}^l$ is a $(k-1)$-flag 
determined by a sequence  $W_{k-1}^l, V_k, \dots, V_d$,
such that  $W_{k-1}^l$ is a $(k-1)$-dimensional linear subspace
of $V_k$ which is different from $V_{k-1}$. 
Notice that the  first term on the right hand side of \eqref{eq:C3.6.3}
corresponds to one of the $(k-1)$-flags in \eqref{eq:C7.1.1}
being in standard position, possibly after re-choosing the positive
and negative half-spaces of $V_k$. We can assume that
this is the case, since if neither of the $(k-1)$-flags corresponds 
to this term, then $\mu_{A,\Phi_{k-1}}$ must be the zero measure and 
again \eqref{eq:C3.6.3} is true.

If $r=0$ and $k=1$, then there is a
unique $(k-1)$-dimensional linear subspace of $V_k$, 
namely, the subspace $V_{k-1} = \{0\}$. Hence
in this case there are no $(k-1)$-dimensional linear subspaces
which are different from $V_{k-1}$, so the sum on the
right hand side of \eqref{eq:C3.6.3} is empty. Thus
we obtain that $ -2 \pi i \xi_k \, \ft{\mu}_{A,\Phi_k} (\xi) =
\ft{\mu}_{A,\Phi_r} (\xi) $ for every $\xi \in \R^d$,
which in particular implies \eqref{eq:C3.6.2}.

If $k = r+1$ and $r \geq 1$, then we apply \lemref{lemC3.14} to each one
of the $(k-1)$-flags $\Psi_{k-1}^l$. We may apply the
lemma since the subspace $W_{k-1}^l$ does not coincide
with $V_{k-1}$. We obtain from the lemma that if
$\alpha >0$ is small enough (not depending on $\eta$)
and if $L$ is large enough, then
\begin{equation}
\label{eq:C3.14.5}
| \ft{\mu}_{A,\Psi_{k-1}^l} (\xi)  | < N^{-1} \, \eta, \quad \xi \in K(r,\alpha,L,\delta),
\end{equation}
for all $1 \leq l \leq N$. Then
\eqref{eq:C3.6.3}, \eqref{eq:C3.14.5} and the estimate
$|\dotprod{\sigma^l}{\vec{e}_k}| \leq 1$ imply \eqref{eq:C3.6.2}.

Finally, it remains to prove the lemma  in the case
when $r+2 \leq k \leq d$. We do this by induction on $k$.
We  multiply both sides of \eqref{eq:C3.6.3}
by the terms $-2 \pi i \xi_j$ $(r+1 \leq j \leq k-1)$,
and obtain
\begin{align}
\label{eq:C3.6.4}
\ft{\mu}_{A,\Phi_k} (\xi) \prod_{j=r+1}^{k} (-2 \pi i \xi_j)  
&= \ft{\mu}_{A,\Phi_{k-1}} (\xi) \prod_{j=r+1}^{k-1} (-2 \pi i \xi_j)  \\
\label{eq:C3.6.5}
&+ \sum_{l=1}^{N} \dotprod{\sigma^l}{\vec{e}_k} \,
\ft{\mu}_{A,\Psi_{k-1}^l} (\xi)  \prod_{j=r+1}^{k-1} (-2 \pi i \xi_j),
\end{align}
By the inductive hypothesis, the right hand side
of \eqref{eq:C3.6.4} satisfies
\begin{equation}
\label{eq:C3.14.4}
\Big| 
\Big( \ft{\mu}_{A,\Phi_{k-1}} (\xi) \prod_{j=r+1}^{k-1} (-2 \pi i \xi_j)  \Big) -
\ft{\mu}_{A,\Phi_r} (\xi) \Big|  < \eta/2, \quad \xi \in K(r,\alpha,L,\delta),
\end{equation}
provided that $\alpha >0$ is small enough (not depending on $\eta$),
$\delta$ is small enough and $L$ is large enough.
Next, we estimate the sum in  \eqref{eq:C3.6.5} by
applying \lemref{lemC3.8} to each one
of the $(k-1)$-flags $\Psi_{k-1}^l$. We may apply the
lemma  since $W_{k-1}^l$ does not coincide
with $V_{k-1}$. We obtain from the lemma that if
$\delta>0$ is small enough,  then
\begin{equation}
\label{eq:C3.14.9}
\Big| \ft{\mu}_{A,\Psi_{k-1}^l} (\xi)  \prod_{j=r+1}^{k-1} (-2 \pi i \xi_j)
\Big|  <   (2N)^{-1} \, \eta, \quad \xi \in K(r,\alpha,L,\delta),
\end{equation}
for all $1 \leq l \leq N$. Then using
\eqref{eq:C3.6.4}, \eqref{eq:C3.6.5},
\eqref{eq:C3.14.4}, \eqref{eq:C3.14.9}  and
the estimate $|\dotprod{\sigma^l}{\vec{e}_k}| \leq 1$,
we obtain that \eqref{eq:C3.6.2} holds.
\end{proof}

\subsection{}
\begin{proof}[Proof of \thmref{thmC3.5}]
We apply \lemref{lemC3.6} with $k=d$. If $\Phi_d$
is a $d$-flag, then the measure
$\mu_{A,\Phi_d}$ is equal to ${\Vol_d}|_A$
(that is, the Lebesgue measure restricted to $A$).
In particular we have
$\ft{\mu}_{A,\Phi_d} = \ft{\1}_A$,
so the condition \eqref{eq:C3.5.2} is a special case
of \eqref{eq:C3.6.2} obtained when $k=d$.
Hence \thmref{thmC3.5} is just a special case of
\lemref{lemC3.6}.
\end{proof}

\begin{remark}
The above proof of \thmref{thmC3.5}  yields a quantitative
estimate on how small should $\delta$ be, and how large should
$L$ be, in order that \eqref{eq:C3.5.2} becomes valid. Indeed,
it can be inferred from the proof that 
there is a constant $c = c(A, \Phi_r) >0$ 
such that \eqref{eq:C3.5.2} is true if
$\delta = c \eta$ and $L = (c \eta)^{-1}$.
\end{remark}


\section{Auxiliary lemmas}

In this section we prove two auxiliary lemmas 
needed for the proof of  \thmref{thmA3.1}.

\subsection{}

\begin{lem}
\label{lemC4.1}
Let $A$ be a polytope in $\R^d$, and let $\Phi_r$ be an $r$-flag 
in standard position $(1 \leq r \leq d-1)$.
Then the function $\ft{\mu}_{A,\Phi_r}$ has the form
\begin{equation}
\label{eq:C4.1.2}
\ft{\mu}_{A,\Phi_r} (\xi) = \sum_{k=1}^{N}
\varphi_k(\xi_1, \xi_2, \dots, \xi_r)
\exp\Big( - 2 \pi i \sum_{j=r+1}^{d} \tau_{k,j} \, \xi_j \Big), \quad \xi \in \R^d,
\end{equation}
where $\tau_{k,j}$ are real numbers, and
$\varphi_k$ are continuous functions on $\R^r$ vanishing at infinity.
\end{lem}

\begin{proof}
Let $V_r, V_{r+1}, \dots, V_d$ be the linear subspaces   given by
\eqref{eq:C3.9.1}, and suppose that
$F_r$ is an $r$-dimensional face of $A$ that is parallel 
to the subspace $V_r$. Then there are real numbers 
$\tau_{r+1}, \tau_{r+2}, \dots, \tau_{d}$ such that 
\[
F_r \subset \{x \in \R^d: x_{r+1} = \tau_{r+1}, \; 
x_{r+2} = \tau_{r+2}, \; \dots, \; x_{d} = \tau_{d} \}.
\]
The Fourier transform of the measure $\sigma := {\Vol_r}|_{F_r}$
(the $r$-dimensional volume measure restricted to $F_r$)
is therefore given by
\begin{equation}
\label{eq:C4.1.3}
\ft{\sigma}(\xi) = \varphi(\xi_1, \xi_2, \dots, \xi_r)
\exp\Big( - 2 \pi i \sum_{j=r+1}^{d} \tau_{j} \, \xi_j \Big),
 \quad \xi \in \R^d,
\end{equation}
where the function $\varphi$ is the Fourier transform of the indicator
function of the polytope in $\mathbb R^{r}$ obtained by 
projecting the face $F_r$ on the $(x_1,x_2,\dots,x_r)$ coordinates.
In particular, $\varphi$ is a continuous function on $\R^r$ vanishing at infinity.

Now the measure $\mu_{A,\Phi_r}$ is a linear combination
(with $\pm 1$ coefficients)
of measures of the form ${\Vol_r}|_{F_r}$, where
$F_r$ belongs to a sequence of faces
$F_r \subset F_{r+1} \subset \cdots \subset F_d$
such that $F_j$ is a $j$-dimensional face of $A$
which is parallel  to $V_j$ $(r \leq j \leq d-1)$. Hence the 
 Fourier transform $\ft{\mu}_{A,\Phi_r}$
of the measure $\mu_{A,\Phi_r}$ is a linear combination
of functions of the form \eqref{eq:C4.1.3}.
This implies that $\ft{\mu}_{A,\Phi_r}$ has the
form \eqref{eq:C4.1.2} as claimed.
\end{proof}

\subsection{}

\begin{lem}
\label{lemC4.7}
Let $p(t)$ be a trigonometric polynomial given by
\begin{equation}
\label{eq:C4.7}
p(t) = \sum_{k=1}^{N} c_k e^{2 \pi i \tau_k t}  \quad (t \in \R)
\end{equation}
where $\tau_k$ are real numbers, and $c_k$ are complex numbers. 
For any $\eta > 0$ there exists a relatively dense set $T \subset \R$,  such that
$|p(t' - t) - p(0)| < \eta$ for any two elements $t, t' \in T$.
\end{lem}

We give two proofs, one relies on the theory of almost periodic
functions (in the same spirit as in \cite{KP02}), while the other 
on a result from dynamical systems.

\begin{proof}[First proof of \lemref{lemC4.7}]
The trigonometric polynomial  $p$ is  a linear combination of periodic functions, 
and so it is an almost periodic function, see for instance \cite[Section VI.5]{Kat04}.
According to the definition of an almost periodic function, this implies
that given $\eta>0$ there exists
a relatively dense set $T \subset \R$ such that
\[
\sup_{x \in \R} |p(x+t) - p(x)| < \eta/2, \quad t \in T.
\]
Then for any two elements $t, t' \in T$ we have
\begin{align}
& \nonumber  |p(t'-t) - p(0)| \leq 
\sup_{x \in \R} |p(x+t') - p(x+t)| \\
& \qquad  \leq
\sup_{x \in \R} |p(x+t') - p(x)|  +\sup_{x \in \R} |p(x+t) - p(x)|
< \eta. \qedhere
\end{align}
\end{proof}

\begin{proof}[Second proof of \lemref{lemC4.7}]
For $\delta>0$, let $T(\delta) = T(\delta; \tau_1, \dots, \tau_N)$ denote the set of integers $t$ for which the condition $\dist( \tau_k t, \Z) < \delta$ holds for all $1 \leq k \leq N$. Then $T(\delta)$ is a relatively dense set, see for instance \cite[Theorem 1.21]{Fur81}.
For any two elements $t, t' \in T(\delta)$ we have
\[
|e^{2 \pi i \tau_k (t' - t)} - 1| \leq 2 \pi \dist( \tau_k (t' - t), \Z) <  4 \pi \delta
\quad  (1 \leq k \leq N),
\]
and therefore
\[
|p(t'-t) - p(0)| \leq \sum_{k=1}^{N} |c_k| \cdot 
|e^{2 \pi i \tau_k (t' - t)} - 1|
\leq 4 \pi \delta \sum_{k=1}^{N} |c_k| .
\]
Hence if $\delta = \delta(p, \eta)$ is chosen  sufficiently small, this implies that
$|p(t' - t) - p(0)| < \eta$.
\end{proof}

\section{Proof of \thmref{thmA3.1}}

We now give the proof of \thmref{thmA3.1} using the results obtained above.
The proof strategy extends the one that was introduced in \cite{KP02}
and further developed in \cite[Section 4]{GL17}.

\subsection{}
Let $A$ be a spectral polytope in $\R^d$, and let $\Phi_r$
 be an $r$-flag $(1 \leq r \leq d-1)$. We must show that $H_{\Phi_r}(A) = 0$.
By applying an invertible linear transformation, we may assume that 
$\Phi_r$ is in standard position.

Suppose to the contrary that $H_{\Phi_r}(A) \neq 0$. Choose a number $\eta$ such that
\begin{equation}
\label{eq:C5.2.1}
0 < 3 \eta < |H_{\Phi_r}(A)|.
\end{equation}
According to \thmref{thmC3.5} we can find $\alpha$,  $\delta$ and $L$ such that
\eqref{eq:C3.5.2} holds. Let $v = v(r,\delta)$ be the vector in $\R^d$ given by
\begin{equation}
\label{eq:C5.5}
v := \sum_{j=r+1}^{d} \delta^{d-j} \, \vec{e}_j ,
\end{equation}
and define 
\begin{equation}
\label{eq:C5.6}
p(t) := \ft{\mu}_{A,\Phi_r} (t v), \quad t \in \R.
\end{equation}
By \lemref{lemC4.1}, the function $\ft{\mu}_{A,\Phi_r}$ is of the form
\eqref{eq:C4.1.2}, and so we have
\begin{equation}
\label{eq:C5.7}
p(t) =  \sum_{k=1}^{N}
\varphi_k(0, 0, \dots, 0)
\exp\Big(- 2 \pi i t \sum_{j=r+1}^{d} \tau_{k,j} \, \delta^{d-j}  \Big).
\end{equation}
Hence $p(t)$ is a trigonometric polynomial of the form
\eqref{eq:C4.7}.
By \lemref{lemC4.7} there is a relatively dense set $T \subset \R$  such that
\begin{equation}
\label{eq:C5.7.2}
|p(t' - t) - p(0)| < \eta, \quad t, t' \in T. 
\end{equation}

Since the function
$\ft{\mu}_{A,\Phi_r}$ is uniformly continuous on $\R^d$ (being the Fourier
transform of a finite measure), there is $\eps>0$
such that 
\begin{equation}
\label{eq:C5.7.1}
|\ft{\mu}_{A,\Phi_r} (\xi') - \ft{\mu}_{A,\Phi_r} (\xi)| < \eta
\quad
\text{whenever} \quad  \xi, \xi' \in \R^d, \; |\xi'-\xi| < 2\eps.
\end{equation}

Define
\begin{equation}
\label{eq:C5.8}
E := \{t v + w : t \in T, \, w\in \R^d, \, |w|<\eps\}.
\end{equation}
Then the set $E$ consists of the union of open balls of radius $\eps$ centered
at the points of the form $tv$ $(t \in T)$. These points
constitute a relatively dense subset of
the line spanned by the vector $v$. 

\subsection{}
We now claim that
\begin{equation}
\label{eq:C5.9}
|\ft{\mu}_{A,\Phi_r} (\xi)| > \eta, \quad \xi \in E-E.
\end{equation}
Indeed, let $\xi$ be a point in $E-E$. Then we  may write $\xi = (t'-t)v + w$,
where $t, t' \in T$ and $|w|<2\eps$. Hence using
\eqref{eq:C5.6}, \eqref{eq:C5.7.2}, \eqref{eq:C5.7.1}
it follows that
\begin{equation}
\label{eq:C5.10}
|\ft{\mu}_{A,\Phi_r} (\xi)| > |\ft{\mu}_{A,\Phi_r} ((t'-t)v)| - \eta
= |p(t'-t)| - \eta >  |p(0)| - 2\eta.
\end{equation}
Note that
\begin{equation}
\label{eq:C5.11}
p(0) = \ft{\mu}_{A,\Phi_r} (0) = \int d\mu_{A,\Phi_r} = 
H_{\Phi_r}(A).
\end{equation}
Hence \eqref{eq:C5.2.1}, \eqref{eq:C5.10} and \eqref{eq:C5.11} imply
that \eqref{eq:C5.9} holds as claimed.

\subsection{}
For each $h >0$, we let $S(h)$ denote the cylinder of radius $h$ along the line spanned
by the vector $v$, that is,
\[
S(h) := \{t v + w \,:\, t\in\mathbb R, \; w \in\mathbb R^d, \; |w|<h\}.
\]
Notice that
\begin{equation}
\label{eq:C5.12}
E - E \subset S(2\eps).
\end{equation}
It is straightforward to check, using \eqref{eq:C5.5}, that there is $R>0$ such that
\begin{equation}
\label{eq:C5.13}
S(2\eps) \setminus B_R \subset K(r,\alpha,L,\delta),
\end{equation}
where $B_R$ denotes the open ball of radius $R$ centered at the origin.

\subsection{}
Let $\Lambda$ be a spectrum for $A$. 
 We claim that for any $\tau\in\mathbb R^d,$ if $\lambda,\lambda'$ are two points in
 $\Lambda\cap(E+\tau)$, then $|\lambda'-\lambda| < R.$ Indeed, if not,
 then it follows from \eqref{eq:C5.12}, \eqref{eq:C5.13} that
\[
\lambda'-\lambda \in (E-E) \setminus B_R  \subset K(r,\alpha,L,\delta).
\]
On the other hand, by \eqref{eqP1.2}  we have $\ft{\1}_{A}(\lambda'-\lambda) = 0$,
 hence \eqref{eq:C3.5.2} implies that we must have
$| \ft{\mu}_{A,\Phi_r} (\lambda'-\lambda) | < \eta$.
However this is not possible, due to  \eqref{eq:C5.9}.

Since $\Lambda$ is a uniformly discrete set, it follows  that
 $\Lambda\cap(E+\tau)$ is a finite set, for every $\tau\in\mathbb R^d$.
Since $\Lambda$ is a relatively dense set, there is $M>0$ such that
 every ball of radius $M$ intersects $\Lambda$.
The cylinder $S(M)$ 
can be covered by a finite number of translates of $E$,
hence $\Lambda \cap S(M)$ is also a
finite set. It follows that $S(M)$ must contain a ball of radius $M$
free from points of $\Lambda$, a contradiction.
\thmref{thmA3.1} is thus proved.
\qed


\section{Remark}

The assumption in \thmref{thmA3.1} (and in \thmref{thmA3.2})
that the polytope $A$ is spectral, was used only in order to know that
there is a relatively dense set of frequencies $\Lambda \subset \R^d$ such that
the exponential system  $E(\Lambda)$ is orthogonal
in the space $L^2(A)$. Hence the result remains valid
under this weaker assumption. In other words, we have actually proved
the following more general version of the result:

\begin{thm}
\label{thmA4.1}
Let $A$ be a polytope in $\R^d$ (not necessarily convex or connected).
Assume that there is a relatively dense set $\Lambda \subset \R^d$ such that
the exponential system  $E(\Lambda)$ is orthogonal
in the space $L^2(A)$. Then $H_{\Phi}(A) = 0$
for every $r$-flag $\Phi$ $(1 \leq r \leq d-1)$.
As a consequence, $A$ is equidecomposable by translations
to a cube of the same volume.
\end{thm}

In the special case when the polytope $A$ is convex, the conclusion
implies that $A$ must be centrally symmetric and have centrally symmetric facets.
This recovers a result stated in \cite[Theorem 5.5]{GL18}.


\end{document}